\documentclass[12pt,a4paper,twoside]{amsart}
\usepackage[utf8x]{inputenc}
\usepackage[T1]{fontenc}
\usepackage{eucal}
\usepackage{lmodern}
\usepackage{amssymb}
\usepackage{amsmath}
\usepackage{amsthm}

\title{The density of twins of $k$-free numbers}
\date{}
\author{Rainer Dietmann}
\address{Department of Mathematics, Royal Holloway University of London,
Egham TW20 0EX, UK}
\email{Rainer.Dietmann@rhul.ac.uk}
\author{Oscar Marmon}
\address{Mathematisches Institut \\ Georg-August-Universität Göttingen \\ Bunsenstr.~ 3-5 \\ 37073 Göttingen \\ Germany}
\email{omarmon@uni-math.gwdg.de}

\begin{document}

\renewcommand{\bf}{\mathbf}
\newtheorem{thm}{Theorem}
\newtheorem{lemma}{Lemma}
\theoremstyle{remark}
\newtheorem*{rem*}{Remark}
\newcommand{\epsi}{\varepsilon}
\newcommand{\ZZ}{\mathbb{Z}}
\newcommand{\RR}{\mathbb{R}}
\newcommand{\NN}{\mathbb{N}}
\newcommand{\CC}{\mathbb{C}}
\newcommand{\xx}{\mathbf{x}}
\newcommand{\yy}{\mathbf{y}}
\newcommand{\cN}{\mathcal{N}}

\newcommand{\en}{\mathbb{N}}
\newcommand{\qu}{\mathbb{Q}}
\newcommand{\zet}{\mathbb{Z}}
\newcommand{\ce}{\mathbb{C}}
\newcommand{\er}{\mathbb{R}}
\newtheorem{theorem}{Theorem}
\newtheorem{corollary}{Corollary}

\begin{abstract}
For $k \geq 2$, we consider the number $A_k(Z)$ of positive integers $n \leq Z$ such that both $n$ and $n+1$ are $k$-free. 
We prove an asymptotic formula $A_k(Z) = c_k Z + O(Z^{14/(9k)+\epsilon})$, where the error term improves upon previously known estimates. The main tool used is the approximative determinant method of Heath-Brown.
\end{abstract}

\maketitle

\section{Introduction}

Let $k \geq 2$ be a natural number. A positive integer is called $k$-free if it is not divisible by the $k$-th power of any prime. It is well known that the set of $k$-free numbers has positive density. Indeed, denoting by $\mu_k(n)$ the characteristic function for the set of $k$-free numbers, 
\[
\mu_k(n) = \begin{cases}
           0 & \text{ if } p^k \mid n \text{ for some prime } p\\
           1 & \text{ otherwise},
           \end{cases}
\]
it is easy to prove the asymptotic formula
\[
\sum_{n \leq Z} \mu_k(n) = \frac{1}{\zeta(k)} Z + O(Z^{1/k}). 
\]

More generally, let $A_k(Z)$ be the number of positive integers $n \leq Z$ such that both $n$ and $n+1$ are $k$-free, that is,
\begin{equation*}
\label{eq:A_k_identity}
A_k(Z) = \sum_{n \leq Z} \mu_k(n)\mu_k(n+1). 
\end{equation*}
Our main result is an asymptotic formula for $A_k(Z)$.

\begin{thm}
\label{thm:main}
We have
\begin{equation}
\label{eq:asymptoticformula}
A_k(Z) = c_k Z + O_{k,\epsi}\left(Z^{\frac{14}{9k}+\epsi}\right)
\end{equation}
for any $\epsi>0$, where
\[
c_k = \prod_p\left(1-\frac{2}{p^k}\right). 
\]
\end{thm}

By elementary methods, one may obtain \eqref{eq:asymptoticformula} with the error term replaced by $O(Z^{2/(k+1)+\epsi})$. Such an asymptotic formula has been known at least since the 1930's (see \cite{Bruedern-Perelli-Wooley00} for a discussion of early references). We shall refer to $2/(k+1) + \epsi$ as the trivial exponent. In the case $k=2$, Heath-Brown \cite{Heath-Brown84} improved the exponent $2/3+\epsi$ to $7/11+\epsi$, using the so-called square sieve. Brandes \cite{Brandes09} adapted this method to arbitrary $k$, obtaining an improvement upon the trivial exponent which is of order $1/k^2$ as $k\to \infty$ (see \cite{Brandes14} for a corrected value of the exponent appearing in \cite{Brandes09}). In a recent preprint, Reuss \cite{Reuss12} gives substantial improvements for small values of $k$, proving the asymptotic formula \eqref{eq:asymptoticformula} with error term $O(Z^{\omega(k)+\epsi})$, where in particular $\omega(2) \approx 0.578$ and $\omega(3) \approx 0.391$. However, whereas in previous results, the exponent approaches the trivial one as $k\to \infty$, the error term in Theorem \ref{thm:main} exhibits a saving of order $1/k$ in the exponent. Our result improves 
upon previously known bounds for $k \geq 6$.

For technical reasons, we shall work with the quantity $A^*_k(Z) = A_k(2Z)-A_k(Z)$ rather than $A_k(Z)$ itself. We shall prove the asymptotic formula
\[
A^*_k(Z) = c_k Z + O_\epsi\left(Z^{\frac{14}{9k}+\epsi}\right),
\]
from which \eqref{eq:asymptoticformula} follows by dyadic summation. (Here, and henceforth in the paper, we suppress the dependence on $k$ in any implied constants.) The proof of this asymptotic formula relies upon an estimate for the density of solutions to a certain Diophantine equation. 

Our initial considerations follow the treatment in \cite{Heath-Brown84}. Using the relation 
\[
\mu_k(n) = \sum_{x^k \mid n} \mu(x), 
\]
we have
\[
A^*_k(Z) = \sum_{Z < n \leq 2Z} \mu_k(n)\mu_k(n+1) = \sum_{x,y} \mu(x)\mu(y) M(x,y,Z),  
\]
where $M(x,y,Z)$ is the number of positive integers $Z < n \leq 2Z$ such that $x^k\mid n+1$ and $y^k \mid n$. By the Chinese Remainder Theorem we have
\[
M(x,y,Z) = \begin{cases}
            \dfrac{Z}{(xy)^k} + O(1) & \text{if } (x,y) = 1,\\
            0 & \text{otherwise}.       
           \end{cases}
\]

First we consider terms with $xy \leq P$, where $P \in [Z^{1/k},Z]$ is a parameter to be specified at a later stage. We have
\begin{multline*}
\sum_{xy \leq P} \mu(x)\mu(y) M(x,y,Z) = Z \sum_{\substack{xy \leq P\\(x,y) = 1}} \frac{\mu(x)\mu(y)}{(xy)^k} + O\left(\sum_{xy \leq P} 1\right) \\
\begin{aligned}
&= Z \sum_{n=1}^\infty \frac{\mu(n) d(n)}{n^k} + O\left(Z \sum_{n>P} \frac{d(n)}{n^k}\right) + O\left(\sum_{n \leq P} d(n) \right)\\
&= c_k Z + O_\epsi\left(Z^{1+\epsi} P^{-(k-1)} \right) + O_\epsi\left(P^{1+\epsi}\right),
\end{aligned}
\end{multline*}
where both error terms are bounded by $O_\epsi\left(PZ^\epsi\right)$, by our assumption on $P$.

We partition the remaining range for $(x,y)$ into $O(\log Z)^2$ boxes of the form $(X,2X]\times(Y,2Y]$. The contribution from each of these may be bounded by
\[
\sum_{\substack{X<x\leq 2X\\ Y<y\leq 2Y}} M(x,y,Z) \leq N(X,Y,Z),
\]
where $N(X,Y,Z)$ is the number of quadruples $(a,b,x,y)\in\NN^4$ satisfying
\begin{equation}
\label{eq:dioph}
\begin{gathered}
ax^k-by^k=1, \\
X< x \leq 2X, \quad Y< y\leq 2Y \quad \text{and} \quad Z< by^k\leq 2Z.
\end{gathered}
\end{equation}
Our preliminary considerations may thus be summarized in the following result.

\begin{lemma}
\label{lem:cutoff}
For any $P \in [Z^{1/k},Z]$, we have
\begin{align*} 
A^*_k(Z) -c_k Z \ll_\epsi Z^\epsi\left(P + \max_{XY \gg P} N(X,Y,Z)\right).
\end{align*}
\end{lemma}
In section \ref{sec:Determimant}, we provide an estimate of $N(X,Y,Z)$ by means of the determinant method.


\section{Counting solutions to a Diophantine equation}
\label{sec:Determimant}

We shall now derive an upper bound for the quantity $N(X,Y,Z)$ defined above, where we may assume, in view of Lemma \ref{lem:cutoff}, that 
\begin{equation}
\label{eq:assumptions}
\max(X,Y) \ll Z^{1/k} \quad  \text{and} \quad XY \gg Z^{1/k}. 
\end{equation}
We shall also assume that $X \leq Y$, the case $Y \leq X$ being entirely similar.

Like Reuss \cite{Reuss12}, we shall use a new version of the determinant method, first introduced in a recent paper by Heath-Brown \cite{Heath-Brown12}. If the positive integers $a,b,x,y$ satisfy \eqref{eq:dioph} and the above height restrictions, then, putting
\begin{equation}
\label{eq:dehomogenization}
t = \frac ba, \quad s = \frac xy \quad \text{and} \quad v = \frac{1}{ay^k},
\end{equation}
we have $t = s^k - v$, and our new variables satisfy
\[
\frac{X}{Y} \ll s \ll \frac{X}{Y} \quad \text{and} \quad \frac{1}{AY^k}\ll v \ll \frac{1}{AY^k},
\]
where $A=ZX^{-k}$. For a certain integer parameter $Y \leq M \leq Z$, to be chosen at a later stage, we cover the admissible range for $s$ by small intervals $(s_0,s_0+M^{-1}]$. In order for this to make sense, we note that $M \gg Y/X$ provided that $Z\gg 1$. We shall separately count solutions with $s$ confined to each of these $O(MX/Y)$ subintervals.

Thus, let $I=(s_0,s_0+M^{-1}]$ and let $R = \{\xx_1,\dotsc,\xx_J\}$ be the set of solutions $\xx_j = (a_j,b_j,x_j,y_j)$ to \eqref{eq:dioph} such that $s_j = x_j/y_j \in I$. Furthermore, for suitably chosen positive integers $d,e$, let $f_1,\dotsc,f_H$, where $H = (d+1)(e+1)$, be an enumeration of the monomials in four variables that are bihomogeneous of bidegree $(d,e)$, that is $f_i(x_1,x_2;y_1,y_2) = x_1^{\alpha_i} x_2^{\beta_i}y_1^{\gamma_i}y_2^{\delta_i}$, where $\alpha_i+\beta_i = d$ and $\gamma_i + \delta_i = e$. Following the general procedure of the determinant method, our aim is now to show that the matrix with entries $f_i(\xx_j)$, where $1\leq i \leq H$, $1\leq j\leq J$, has rank less than $H$. Indeed, this ensures the existence of a non-zero bihomogeneous polynomial $B(\xx;\yy)$ of bidegree $(d,e)$ vanishing at every $\xx_j$. As in \cite{Heath-Brown12}, one argues that the coefficients of $B$ may be chosen to have size $O(Z^\kappa)$ for some natural number $\kappa$ depending only on $d$ and $e$.

If $J < H$, the above assertion is trivially true. Otherwise, we choose a subset of $R$ of cardinality $H$ --- without loss of generality we may take $\{\xx^{(1)},\dotsc,\xx^{(H)}\}$ --- and prove that the corresponding 
$H\times H$-subdeterminant
\[
\Delta_1 = \det(f_i(\xx_j))_{1\leq i,j \leq H}
\]
vanishes. Note that, since the value of $\Delta_1$ is an integer, it suffices to prove that $|\Delta_1| < 1$.

Defining $s_j,t_j,v_j$ in the obvious way according to \eqref{eq:dehomogenization}, we have
\begin{equation}
\label{Delta_1}
\Delta_1 = \prod_{j=1}^H a_j^d y_j^e \Delta_2 \ll A^{dH}Y^{eH}|\Delta_2|,
\end{equation}
where $\Delta_2 =  \det \big( f_i(1,t_j,s_j,1) \big) = \det \big( t_j^{\alpha_i},s_j^{\beta_i} \big)$. We may now write $s_j = s_0 + u_j$ and define new polynomials
\[
g_i(u,v):=f_i(1,(s_0+u)^k-v,s_0+u,1).
\]
In this notation, we have
\[
\Delta_2 =  \det \big( g_i(u_j,v_j) \big)_{1\leq i,j \leq H}.
\]
Putting $V = AY^k$, we have $|u_j|\ll M^{-1}$ and $|v_j| \ll V^{-1}$. Furthermore, we note that the polynomials $g_i$ have degree at most $kd + e$ and coefficients of size $O_{k,d,e}(1)$.

We shall now estimate the determinant $\Delta_2$ using Lemma 3 in \cite{Heath-Brown09}. (Unless explicitly stated otherwise, the implied constants occurring in the following calculations are uniform in $d$ and $e$.) Thus, let $m_1,m_2,\dotsc$ be all possible monomials in two variables, enumerated in such a way that $1=M_1 \geq M_2 \geq \dotsb $, where $M_i := m_i(M^{-1},V^{-1})$.
Then, according to Heath-Brown's lemma, we have
\begin{equation}
\label{Delta_2}
\Delta_2 \ll_{H} \prod_{i=1}^H M_i.
\end{equation}
Put $W = M_H^{-1}$. Then the factor $M^{-j}V^{-l}$ occurs in the product $\prod M_i$ if and only if $M^j V^l \leq W$. Furthermore, our assumptions above imply that
\[
1 \ll \frac{\log V}{\log Z} \ll 1, \quad 1 \ll \frac{\log M}{\log Z} \ll 1 \quad \text{and} \quad 1 \ll \frac{\log V}{\log M} \ll 1.
\]

Thus, letting $T$ be the set of $(j,l) \in \NN^2$ that satisfy
\[
j \log M + l \log V \leq \log W,
\]
it follows that
\begin{align*}
H &= \#T = \frac{(\log W)^2}{2\log M \log V} + O\left(\frac{\log W}{\log Z}\right) + O\left(1\right).
\end{align*}
By our assumptions, we have $\log W \gg \log Z$, so we may deduce that
\begin{equation}
\label{eq:log W}
\log W = H^{1/2}(2\log M\log V)^{1/2} + O(\log Z).
\end{equation}
Furthermore, we have
\begin{align*}
\log \left(\prod_{i=1}^H M_i\right) &= -\sum_{(j,l) \in T}(j \log M + l\log V) \\
&= -\frac{(\log W)^3}{3\log M \log V} + O\left(\frac{(\log W)^2}{\log Z}\right),
\end{align*}
and thus, by \eqref{eq:log W},
\[
\log \left(\prod_{i=1}^H M_i\right) = -\frac{2\sqrt{2} H^{3/2}}{3}(\log M\log V)^{1/2} + O(H \log Z).
\]
It follows that
\[
\log |\Delta_2| \leq O_H(1) - \frac{2\sqrt{2} H^{3/2}}{3}(\log M\log V)^{1/2} + O(H \log Z),
\]
so, in view of the estimate \eqref{Delta_1}, we need to show that
\[
dH \log A + eH \log Y \leq \frac{2\sqrt{2} H^{3/2}}{3}(\log M\log V)^{1/2} - O_H(1) - O(H \log Z).
\]

To this end, we begin by fixing the ratio between the degrees $d$ and $e$, putting $e = \lfloor d \log A/\log Y \rfloor$. By our earlier assumptions, we then have $d \ll e \ll d$. It now suffices to show that
\begin{equation}
\label{eq:dlogA}
d \log A \leq d \frac{\sqrt{2}}{3}\left(\frac{\log A}{\log Y}\right)^{1/2} (\log M\log V)^{1/2} - O_d(1) - O(\log Z).
\end{equation}
If, for some number $\delta >0$, we have
\begin{equation}
\label{eq:M-bound1}
\frac{\sqrt{2}}{3}\left(\frac{\log A}{\log Y}\right)^{1/2} (\log M\log V)^{1/2} \geq (1+\delta) \log A,
\end{equation}
then \eqref{eq:dlogA} will indeed hold as soon as $Z \gg_\delta 1$ and $d \gg_\delta 1$. The condition \eqref{eq:M-bound1} may be rewritten as
\[
\log M \geq \frac{9}{2}(1+\delta)^2 \frac{\log A \log Y}{\log V}.
\]
Redefining $\delta$, and noting that $V \geq Z$, we may summarize our findings as follows (cf. \cite[Lemma 1]{Heath-Brown12}).

\begin{lemma}
\label{lem:auxiliaryform}
If $M$ satisfies
\begin{equation}
\label{eq:logM}
\log Z \geq \log M \geq \max\left\{ \frac92(1 + \delta)\frac{\log A \log Y}{\log Z}, \log Y \right\}
\end{equation}
for a given $\delta >0$, then the following holds. For any interval $I = (s_0,s_0+M^{-1}]$ there is a non-zero bihomogeneous polynomial $B_I(\xx;\yy)$ such that
\[
B_I(a,b;x,y) = 0
\]
for every solution to \eqref{eq:dioph} such that $x/y \in I$. Moreover, $B_I$ has total degree $O_\delta(1)$ and coefficients of size $O(Z^\kappa)$, for some constant $\kappa = \kappa(\delta)$.
\end{lemma}

Our aim is now to estimate the contribution to \label{comment7} $N(X,Y,Z)$ from each interval $I$. As in \cite{Heath-Brown12}, we may assume that the polynomial $B_I(\xx;\yy)$ is absolutely irreducible, with coefficients of size at most $O(Z^\kappa)$. If $I = (s_0,s_0+M^{-1}]$, then the points $(a,b,x,y)$ of interest certainly satisfy 
\begin{equation}
\label{eq:parallelogram}
|y|\leq 2Y, \quad |x-s_0y|\leq 2YM^{-1}. 
\end{equation}
Thus, we now wish to bound the number $N_I$ of points $(a,b,x,y) \in \ZZ^4$ in the region defined by \eqref{eq:parallelogram}  that satisfy the equations
\begin{align}
\label{eq:original}
ax^k-by^k & = 1 \text{ and} \\
\label{eq:B_I}
B_I(a,b;x,y) & = 0.
\end{align}

We shall make a coordinate change in order to take advantage of the thinness of the parallelogram \eqref{eq:parallelogram}. Following \cite{Heath-Brown12}, we consider the lattice
\[
\Lambda_I = \left\{\left(\frac{M}{2Y}(x-s_0y),\frac{1}{2Y}y\right);(x,y) \in \ZZ^2\right\},
\]
with determinant $\det(\Lambda_I) = M/(4Y^2)$. Much as in \cite{Heath-Brown12}, we choose $\bf{g}^{(1)},\bf{g}^{(2)} \in \Lambda_I$ so that $|\bf{g}^{(1)}|$ is minimal among non-zero vectors of $\Lambda_I$, and $|\bf{g}^{(2)}|$ is minimal among vectors not parallel to $\bf{g}^{(1)}$. Then $\bf{g}^{(1)},\bf{g}^{(2)} \in \Lambda_I$ form a basis for $\Lambda_I$. Furthermore, we have $|\bf{g}^{(1)}||\bf{g}^{(2)}| \asymp \det(\Lambda_I)$, and if $\xx \in \Lambda_I$ is expressed in this basis as $\xx= \lambda_1 \bf{g}^{(1)} + \lambda_2 \bf{g}^{(2)}$, then $|\lambda_i| \ll |\xx|/|\bf{g}^{(i)}|$.  Thus, taking $L_i$ to be suitable multiples of $|\bf{g}^{(i)}|^{-1}$ for $i = 1,2$, we have
\(
\lambda_1 \bf{g}^{(1)} + \lambda_2 \bf{g}^{(2)} \in [-1,1]^2 \text{ only if }  |\lambda_i| \leq L_i,
\)
and furthermore
\[
L_1 \gg L_2, \quad L_1L_2 
\asymp Y^2M^{-1}.  
\]

We have $\Lambda_I = \mathsf{L}\,\ZZ^2$, where 
\[
\mathsf{L} = \begin{pmatrix}
              \frac{M}{2Y} & \frac{-s_0 M}{2Y}\\[0.3em]
              0 & \frac{1}{2Y}
             \end{pmatrix}.
\]
By the above, the vectors $\mathsf{L}^{-1} \bf{g}^{(1)}, \mathsf{L}^{-1}\bf{g}^{(2)}$ constitute a basis for $\ZZ^2$. If the new coordinates $(\lambda_1,\lambda_2)$ are defined by 
$\left(x,y\right) = \lambda_1 \mathsf{L}^{-1} \bf{g}^{(1)} + \lambda_2 \mathsf{L}^{-1} \bf{g}^{(2)}$, 
we may now bound $N_I$ from above by the number of solutions $(a,b,\lambda_1,\lambda_2) \in \ZZ^4$ to  
\begin{equation}
\label{eq:lambda-conditions}
F(a,b,\lambda_1,\lambda_2) = 1, \quad G(a,b,\lambda_1,\lambda_2) = 0,  \quad |\lambda_i| \leq L_i, 
\end{equation} 
where $F$ is bihomogeneous of bidegree $(1,k)$ and $G$ is bihomogeneous of bidegree $(d,e)$, say, and where $F$ and $G$ again have integer coefficients bounded by a power of $Z$. We shall now prove the following estimate.


\begin{lemma}
\label{lem:N_I}
In the above notation, we have
\[
N_I \ll_{\delta,\epsi} Z^\epsi L_1^{1+\epsi}.
\] 
\end{lemma}

The proof of this estimate is divided into different cases according to the value of $d$. Clearly we may assume that $L_1 \geq 1$, as otherwise $N_I$ will vanish. In case $\min(d,e) \geq 1$, Lemma 2 in \cite{Heath-Brown12} then states that the number of solutions to the equation $G(a,b,\lambda_1,\lambda_2) = 0$ satisfying $\gcd(a,b)=\gcd(\lambda_1,\lambda_2) = 1$ and $|\lambda_i| \leq L_i$ is 
\[
O_{d,e,\epsi}\left(L_1^{2/d+\epsi}\Vert G\Vert^\epsi\right). 
\]
This establishes the bound in Lemma \ref{lem:N_I} as soon as $d \geq 2$ and $e \geq 1$. (Indeed, the indivisibility conditions are automatically satisfied on account of the first equation in \eqref{eq:lambda-conditions}.) It remains to settle the cases where $d=0$, $d=1$ or $e=0$.

Assume first that $d=0$, so that $G(a,b,\lambda_1,\lambda_2) = H(\lambda_1,\lambda_2)$, say. Then there are only $O(1)$ possibilities for $(\lambda_1,\lambda_2)$, and thus for $(x,y)$. For fixed $(x,y)$, the number of pairs $(a,b)$ satisfying \eqref{eq:dioph} and $|ax^k|\leq Z$ is
\[
\ll 1 + \frac{Z}{(XY)^k} \ll 1,
\]
by the assumptions in \eqref{eq:assumptions}.

Next, if $e=0$, the equation $G=0$ determines at most $d$ pairs $(a,b)$, and for each such choice, the first equation in \eqref{eq:lambda-conditions} reads $\tilde F(\lambda_1,\lambda_2) = 1$, for some homogeneous polynomial $\tilde F$. We cannot rule out the possibility that $\tilde F$ is a power of a single linear form, but even the trivial bound $O(L_1)$ for the number of solutions $(\lambda_1,\lambda_2)$ suffices for Lemma \ref{lem:N_I}.

Finally, in the case $d=1$, we argue exactly as in \cite{Heath-Brown12}. If we write
\[
G(a,b,\lambda_1,\lambda_2) = a G_1(\lambda_1,\lambda_2) + b G_2(\lambda_1,\lambda_2), 
\]
the condition $G=0$ implies that
\begin{equation}
\label{eq:commonfactor}
G_1(\lambda_1,\lambda_2) = -qb, \quad G_2(\lambda_1,\lambda_2) = qa 
\end{equation}
where the integer $q$ divides the resultant of $G_1$ and $G_2$. As the coefficients of $G$ are bounded by powers of $Z$, we have only $O(Z^\epsi)$ choices for $q$. For each choice, substituting \eqref{eq:commonfactor} into the equation  $F(a,b,\lambda_1,\lambda_2)= 1$ gives a Thue equation $\tilde F(\lambda_1,\lambda_2) = q$, which again can have at most $O(L_1)$ solutions.
This completes the proof of Lemma \ref{lem:N_I}. (Note that the original equation $ax^k-by^k = 1$ was discarded in most cases.)

In view of Lemma \ref{lem:N_I}, the above transformation is most useful when $L_1$ is not too big, that is, when the shortest vector in $\Lambda_I$ is not too short. To sum up the contribution from all the intervals $I$, we thus need to know how often $L_1$ is of a certain size. It is now convenient to assume that the intervals $I$ in the above subdivision are defined by taking $s_0 = z/M$ for an integer $z \ll MX/Y$. In fact, by assuming that $Z \gg 1$, so that $M\gg Y/X$, we may ensure that only values $z > 0$ are needed.

\begin{lemma}
\label{lem:countshortestvectors}
We have $Y/M^{1/2} \ll L_1 \ll Y$. Moreover, the number of intervals $I=(s_0,s_0+M^{-1}]$ for which $L \leq L_1 \leq 2L$ is at most 
\[
O_\epsi\left(Z^\epsi\left(\frac{Y}{L}+\frac{XY}{L^2}\right)\right).
\]
\end{lemma}
 
\begin{proof}
Suppose that $\bf{g}^{(1)} = \left(\frac{M}{2Y}(x_1-s_0y_1),\frac{1}{2Y}y_1\right)$. Clearly we have $|\bf{g}^{(1)}|\gg Y^{-1}$, whence the upper bound for $L_1$. The lower bound follows from the fact that $L_1 \geq L_2$ and $L_1L_2 \gg Y^2M^{-1}$. 

By the definition of $L_1$ we have
\[
L_1(x_1-s_0 y_1) \ll \frac Y M, \quad L_1 y_1 \ll Y. 
\]
Suppose now that $L \leq L_1 \leq 2L$. With $z$ as defined above, it follows that 
\begin{equation}
\label{eq:countshortestvectors}
y_1 z = M x_1 + O\left(\frac{Y}{L}\right). 
\end{equation}
As the left hand side of \eqref{eq:countshortestvectors} is $\ll MX/L$, we must have $x_1 \ll X/L$. If $x_1 = 0$, then by definition of $\bf{g}^{(1)}$ we must have $y_1 = \pm 1$, leaving at most $O(Y/L)$ choices for $z$. 

For each choice of $x_1\neq 0$, there are at most $O(Y/L)$ possible choices for the right hand side of \eqref{eq:countshortestvectors}. Moreover, in this case \eqref{eq:countshortestvectors} implies that
\begin{equation}
\label{eq:y1z_nonzero}
M \ll y_1 z \ll \frac{MX}{L}, 
\end{equation}
and in particular $y_1z \neq 0$. A divisor function estimate now shows that there are $O(Z^\epsi Y/L)$ possible choices for $y_1$ and $z$ for each choice of $x_1\neq 0$. As $x_1 \ll X/L$, where $X/L \gg 1$ by \eqref{eq:y1z_nonzero}, the contribution from intervals $I$ with $x_1 \neq 0$ is $O(Z^\epsi XY L^{-2})$. 
\end{proof}

Combining Lemmas \ref{lem:N_I} and \ref{lem:countshortestvectors}, we see that the total contribution to $N(X,Y,Z)$ from all intervals such that $L \leq L_1 \leq 2L$ is 
\[
O_{\delta,\epsi}\left(Z^\epsi \left(\frac{XY}{L} + Y \right)\right).
\] 
Let us temporarily assume that $M \ll Y^2$. By dyadic summation over the range $Y/M^{1/2} \ll L \ll Y$, we then obtain the estimate
\begin{equation}
\label{eq:N_prov}
N(X,Y,Z) \ll_{\delta,\epsi} Z^\epsi X M^{1/2} + Z^\epsi Y.
\end{equation}
In the case where $M\gg Y^2$, however, the bound \eqref{eq:N_prov} is trivial. Indeed, as shown above, the contribution to $N(X,Y,Z)$ from each fixed pair $(x,y)$ is at most $O(1)$, so we get
\[
N(X,Y,Z) \ll XY \ll XM^{1/2} 
\]
in this case. We have shown the following result.

\begin{lemma}
\label{lem:N}
Under the assumptions \eqref{eq:assumptions} and $X\leq Y$, we have
\begin{equation*}
N(X,Y,Z) \ll_{\delta,\epsi} Z^\epsi X M^{1/2} + Z^\epsi Y,
\end{equation*}
as soon as $M$ satisfies \eqref{eq:logM}.
\end{lemma}

As already remarked, the case $Y \leq X$ may be treated in an entirely similar fashion. (Indeed, upon renaming the variables, this amounts to carrying out the analysis of the present section for the equation $ax^k - by^k = -1$.) Thus, in this case the conclusion of Lemma \ref{lem:N} holds true with $X$ and $Y$ interchanged, and $A$ replaced by $B:=ZY^{-k}$. 

\begin{rem*}
Note that by a direct application of \cite[Lemma 2]{Heath-Brown12}, without the above coordinate transformation, we would have obtained the weaker estimate $N(X,Y,Z) \ll Z^\epsi XM$ in Lemma \ref{lem:N}. 
\end{rem*}


\section{Proof of Theorem \ref{thm:main}}
\label{sec:k-free_proof}

We shall now determine the optimal choice for the parameter $P$, and derive an upper bound for $N(X,Y,Z)$ valid for arbitary $X,Y$ with $XY \gg P$. By our previous remarks, we may assume, without loss of generality, that $X \leq Y$. We may then write $X \approx Z^\alpha$, $Y \approx Z^\beta$ and $P \approx Z^\phi$, where
\[
\alpha \leq \beta \leq \frac1k, \quad \alpha+\beta \geq \phi \geq \frac1k. 
\]
so that the conditions \eqref{eq:assumptions} are satisfied. In order to fulfil the second inequality in \eqref{eq:logM}, we choose $\delta$, depending on $\epsi$, such that $$\frac{9}{2}\delta(1-k\alpha)\beta \leq \epsi,$$ and we take $M \in \NN$ to satisfy 
\[
\max\left\{Z^{\frac{9}{2} (1+\delta) (1-k\alpha)\beta},Z^{\beta}\right\} \leq M \ll \max\left\{Z^{\frac{9}{2}(1+\delta)(1-k\alpha)\beta}, Z^{\beta}\right\}.
\]
Provided that the first inequality in \eqref{eq:logM} also holds, Lemma \ref{lem:N} then gives the estimate
\begin{equation}
\label{eq:N}
\begin{aligned}
N(X,Y,Z) &\ll_\epsi Z^\epsi \left( Z^{\alpha + \frac{9}{4}(1-k\alpha)\beta} + Z^{\alpha + \frac12 \beta} + Z^\beta\right) \\
&\ll_\epsi Z^\epsi \left( Z^{\alpha + \frac{9}{4}(1-k\alpha)\beta} + Z^{3/(2k)} \right). 
\end{aligned}
\end{equation}
Putting $u = k\alpha$, $v = k \beta$ and $w = k \phi$, we are then led to consider the functions
\[
\Phi(u,v) = \frac92(1-u)v \quad \text{and} \quad \Psi(u,v) = u + \frac94\left(1- u\right)v.  
\]
The admissible range for $(u,v)$ is the triangular region $T_w$ defined by
\begin{equation}
\label{eq:region1}
u \leq v \leq 1, \quad u+v \geq w.  
\end{equation}
Provided that $\Phi(u,v) < k$ throughout $T_w$, the condition $M \leq Z$ of \eqref{eq:logM} may certainly be fulfilled by choosing $\delta$ small enough, in which case \eqref{eq:N} yields
\begin{equation*}
N(X,Y,Z) \ll_\epsi Z^\epsi \left( Z^{\psi/k} + Z^{3/(2k)} \right), \quad \text{where} \quad \psi = \max_{(u,v) \in T_w} \Psi(u,v). 
\end{equation*}

To prove Theorem \ref{thm:main}, we shall take $w = 14/9$. We observe that
\begin{align*}
\nabla \Phi &= \left(-\frac92 v, \frac92(1-u) \right) \neq (0,0) \\
\text{and} \quad 
\nabla \Psi &= \left(1 -\frac94 v, \frac94(1-u) \right) \neq (0,0)
\end{align*}
throughout $T_{14/9}$, so the maxima of $\Phi$ and $\Psi$ are attained at the boundary, consisting of the line segments
\begin{gather*}
L_1: \quad v=u, \ \frac79 \leq u \leq 1,\\
L_2: \quad v=1,\ \frac59 \leq u \leq 1,\\
L_3: \quad v=\frac{14}{9} - u,\ \frac59 \leq u \leq \frac79
\end{gather*}
By investigating the behaviour of $\Phi$ and $\Psi$ on these line segments, one may check that both functions in fact attain their maximum at $(5/9,1)$. Thus we indeed have $\Phi(u,v) \leq \Phi(5/9,1) = 2 < k$ throughout $T_{14/9}$, as required, and
\[
\max_{(u,v)\in T_{14/9}} \Psi(u,v) = \Psi(5/9,1) = 14/9.
\]

We conclude that 
\[
N(X,Y,Z) \ll_\epsi Z^{14/(9k) + \epsi}
\]
as soon as $XY \gg Z^{14/(9k)}$. Thus, Theorem \ref{thm:main} now follows from Lemma \ref{lem:cutoff}.

\begin{rem*}
One may improve the exponent $14/9$ slightly by replacing the bound from Lemma \ref{lem:N}, in the case when $Y/X$ is large, with the bound $N(X,Y,Z) \ll Z^{1+\epsi}XY^{-k}$, which may be obtained by a more elementary method (cf. \cite[p. 254]{Heath-Brown84}). However, the saving obtained in this way is small for large $k$. 
\end{rem*}


\bibliographystyle{plain}
\bibliography{ratpoints.bib}

\begin{thebibliography}{1}

\bibitem{Brandes14}
Julia Brandes.
\newblock Sums and differences of power-free numbers.
\newblock in preparation.

\bibitem{Brandes09}
Julia Brandes.
\newblock Twins of $s$-free numbers.
\newblock Diploma Thesis, University of Stuttgart, arXiv:1307.2066, 2009.

\bibitem{Bruedern-Perelli-Wooley00}
J.~Br{\"u}dern, A.~Perelli, and T.~D. Wooley.
\newblock Twins of {$k$}-free numbers and their exponential sum.
\newblock {\em Michigan Math. J.}, 47(1):173--190, 2000.

\bibitem{Heath-Brown84}
D.~R. Heath-Brown.
\newblock The square sieve and consecutive square-free numbers.
\newblock {\em Math. Ann.}, 266(3):251--259, 1984.

\bibitem{Heath-Brown09}
D.~R. Heath-Brown.
\newblock Sums and differences of three {$k$}th powers.
\newblock {\em J. Number Theory}, 129(6):1579--1594, 2009.

\bibitem{Heath-Brown12}
D.R. Heath-Brown.
\newblock {Square-free values of $n^2+1$.}
\newblock {\em Acta Arith.}, 155(1):1--13, 2012.

\bibitem{Reuss12}
Thomas Reuss.
\newblock Pairs of $k$-free numbers, consecutive square-full numbers.
\newblock arXiv:1212.3150, 2012.

\end{thebibliography}

\end{document}